\documentclass[a4paper,oneside,11pt]{article}

\usepackage[english]{babel}
\usepackage[T1]{fontenc}
\usepackage[ansinew]{inputenc}
\usepackage{geometry}
\usepackage{lmodern}

\usepackage{graphicx} 

\usepackage{amsmath} \numberwithin{equation}{section}
\usepackage{amsthm}
\usepackage{amsfonts}
\usepackage{amssymb}
\usepackage{graphics}
\usepackage{color}

		\theoremstyle{plain}
    
			\newtheorem{theorem}{Theorem}[section]
      \newtheorem{lemma}[theorem]{Lemma}
      
     \newtheorem{question}[theorem]{Question} \newtheorem{remark}[theorem]{Remark}
      \theoremstyle{definition}
      \newtheorem{definition}[theorem]{Definition}
        \newtheorem{proposition}[theorem]{Proposition}    
      \theoremstyle{remark}
      
      \newtheorem{example}[theorem]{Example}

\usepackage{enumerate}

\newcommand\with{\ \vrule\ }  
\newcommand\C{\mathbb C}         
\newcommand\R{\mathbb R}         
\newcommand\N{\mathbb N}         
\newcommand\Ha{\mathbb H}        
\newcommand\D{\mathbb D}         
\renewcommand\P{\mathfrak{P}}

\renewcommand\i{\operatorname{i}}        
\renewcommand\Im{\text{Im}}        
\renewcommand\Re{\text{Re}}        

\newcommand{\Landauo}{{\scriptstyle\mathcal{O}}}

\newcommand{\K}{\mathcal K}
\newcommand{\B}{\mathbb B}

\newcommand{\I}{\text{Inf}(\mathbb{B}_n)}
\newcommand{\II}{\text{Inf}(\mathbb{D})}
\newcommand{\III}{\text{Inf}(\mathbb{H})}
\newcommand{\Ho}{\mathcal{H}} 

\usepackage[pdfstartview=FitH]{hyperref}

\begin{document}

\setcounter{section}{0}

\title{Chordal generators and \\the hydrodynamic normalization for the unit ball}
\author{Sebastian Schlei\ss inger\thanks{Supported by the ERC grant  ``HEVO - Holomorphic Evolution Equations'' n. 277691.}}
\date{\today}
\maketitle
\abstract{Let $c\geq0$ and denote by $\K(\Ha,c)$ the set of all infinitesimal generators $G:\Ha\to\C$ on the upper half-plane $\Ha$ such that $\limsup_{y\to\infty}y\cdot |G(iy)|\leq c.$ This class is related to univalent functions $f:\Ha\to\Ha$ with hydrodynamic normalization and appears in the so called chordal Loewner equation. \\
In this paper, we generalize the class $\K(\Ha,c)$ and the hydrodynamic normalization to the Euclidean unit ball in $\C^n$.
 The generalization is based on the observation that $G\in\K(\Ha,c)$ can be characterized by an inequality for the hyperbolic length of $G(z).$ }\\



\tableofcontents
\newpage
\section{Introduction}
\setlength{\parindent}{0pt}
\subsection{One-parameter semigroups}

Let $\mathbb{B}_n=\{z\in\C^n \,|\, \|z\|<1\}$ be the Euclidean unit ball in $\C^n$. In one dimension we write $\D:=\B_1$ for the unit disc.

\begin{definition} \label{semi}
 A continuous one-real-parameter semigroup of holomorphic functions on $\B_n$ is a map $[0,\infty)\ni t\mapsto \Phi_t \in \mathcal{H}(\B_n, \B_n)$ satisfying the following conditions: 
 \begin{itemize}
\item[(1)] $\Phi_0$ is the identity,  
\item[(2)] $\Phi_{t+s}=\Phi_t\circ \Phi_s$ for all $t,s\geq 0,$
\item[(3)] $\Phi_t$ tends to the identity locally uniformly in $\B_n$ when $t$ tends to $0.$ \end{itemize}
\end{definition}

Given such a semigroup $\{\Phi_t\}_{t\geq 0}$ and a point $z\in \B_n,$ then the limit $$G(z):=\lim_{t\to 0}\frac{\Phi_t(z)-z}{t}$$ exists and the vector field $G:\B_n\to\C^n$, called the \textit{infinitesimal generator}\footnote{There is no standard convention in the literature and often $-G$ is called the infinitesimal generator of the semigroup.} of $\Phi_t,$ is a holomorphic function (see, e.g., \cite{MR1174816}). We denote by $\I$ the set of all infinitesimal generators of semigroups in $\B_n$. For any $z\in \B_n,$ the map $w(t):=\Phi_t(z)$ is the solution of the initial value problem \begin{equation}\label{cauchy}\frac{dw(t)}{dt}=G(w(t)),\quad w(0)=z.\end{equation} 

There are various characterizations of holomorphic functions $G:\B_n\to \C^n$ that are infinitesimal generators; 
see \cite{MR2022955} (Section 7.3), \cite{MR2578602} (Theorem 0.2), \cite{MR3184757} (p. 193).\\

The set $\II$, i.e. all infinitesimal generators in the unit disc, can be characterized completely by the Berkson-Porta representation formula (see \cite{MR0480965}):
\begin{equation}\label{BerksonPorta}
\II=\{z\mapsto (\tau-z)(1-\overline{\tau}z)p(z)\with \tau\in\overline{\D}, \; p\in\Ho(\D,\C) \;\text{with}\; \Re(p(z))\geq0 \;\text{for all}\; z\in\D\}.
\end{equation}

\begin{remark}
 
Let $F:\D\to \D$ be a holomorphic self-map. Recall the Denjoy-Wolff theorem (see, e.g., \cite{MR2022955}, Theorem 5.1): If $F$ is not an elliptic automorphism (i.e. an automorphism with exactly one fixed point in $\D$), then there exists one point $\tau \in \overline{\D}$ (the Denjoy-Wolff point of $F$) such that the iterates $F^n$ converge locally uniformly in $\D$ to the constant map $\tau.$\\

 If $\{\Phi_t\}_{t\geq 0}$ is a semigroup on $\D$, then we call $\tau\in \overline{\D}$ the Denjoy-Wolff point of $\{\Phi_t\}_{t\geq 0}$ if $\tau$ is the Denjoy-Wolff point of $\Phi_1$, which is equivalent to $\lim_{t\to\infty}\Phi_t = \tau$ locally uniformly. 
 
If an infinitesimal generator in the unit disc does not generate a semigroup of elliptic automorphisms of $\D,$ then the point $\tau \in \overline{\D}$ from formula (\ref{BerksonPorta}) is exactly the Denjoy-Wolff point of the semigroup.
\end{remark}

There are two special cases of infinitesimal generators in $\D$ that have been studied intensively and turned out to be quite useful in Loewner theory and its applications. 
The two different cases arise from certain normalizations of the Berkson-Porta data $\tau$ and $p$ from formula (\ref{BerksonPorta}). In the \emph{radial} case, one considers those elements $G \in \II$ whose Berkson-Porta data $\tau$ and $p$ satisfy
$$\tau = 0 \qquad \text{and} \qquad p(0)=1,$$
i.e. $G(z)=-zp(z)$. 
\\

This class plays a central role in studying the class $S$ of all univalent functions $f:\D\to\C$ with $f(0)=0,$ $f'(0)=1,$ by the powerful tools of Loewner's theory; see, e.g., \cite{Pom:1975}, Chapter 6. The class of radial generators as well as the class $S$ have been generalized in this context to the polydisc $\D^n$, see \cite{MR1049182, MR1049183}, and to the unit ball $\B_n$, see \cite{graham2003geometric} for a collection of several results and references.\\

The second class, the set of all \emph{chordal} generators\footnote{
Note that there is no standard use of the words ``radial'' and ``chordal'' in the literature. In \cite{MR2719792}, e.g., an element $G\in \II$ is called \emph{radial} if $\tau \in \D$ and chordal if $\tau \in \partial \D.$ } , consists of all $G \in \II$ whose Berkson-Porta data $\tau$ and $p$ satisfy
$$\tau = 1 \qquad \text{and} \qquad \angle \lim_{z\to1}\frac{p(z)}{z-1} \quad \text{is finite}.
$$

The aim of this paper is to introduce a generalization of the chordal class for the unit ball $\B_n$. 

\subsection{The hydrodynamic normalization in one dimension}

Instead of fixing an interior point, like in the class $S$, it can be of interest to investigate univalent self-mappings of $\D$ that fix a boundary point. In this case, one usually passes from $\D$ to the upper half-plane $\Ha=\{z\in\C\with \Im(z)>0\}.$\\
A class of such mappings that is easy to describe and that appears in several applications is the set of all univalent mappings $f:\Ha\to\Ha$ that fix the boundary point $\infty$ and have the so called \emph{hydrodynamic normalization}.
Basic properties of this class can be found in \cite{MR1201130}, see also \cite{MR2107849} and \cite{MR2719792}. One of its main applications is the chordal Loewner equation, see \cite{AbateBracci:2010}, Section 4, for further references. \\

A univalent function $f:\Ha\to\Ha$ has \emph{hydrodynamic normalization} (at $\infty$) if $f$ has the expansion
$$f(z)=z-\frac{c}{z}+\gamma(z),$$
where $c\geq0$, which is usually called \emph{half-plane capacity}, and $\gamma$ satisfies $\angle\lim_{z\to \infty}z\cdot\gamma(z)=0.$\\
We denote by $\P$ the set of all these functions. Then $\P$ is a semigroup and the functional $l:\P\to [0,\infty)$, $l(f)=c,$ is additive: If $f_1,f_2\in \P,$ then $f_1 \circ f_2\in \P$ and $l(f_1 \circ f_2)= l(f_1) + l(f_2).$

\begin{remark} Let $f\in \P$ with $l(f)=c.$ If we transfer $f$ to the unit disc by conjugation by the Cayley transform, then we obtain a function $\tilde{f}:\D\to\D$ having the expansion
$$\tilde{f}(z)=z-\frac{c}{4}(z-1)^3+\tilde{\gamma}(z),$$
where $\angle\lim_{z\to1}\frac{\tilde{\gamma}(z)}{(z-1)^3}=0.$
\end{remark}

If $\{\Phi_t\}_{t\geq 0}$ is a one-real-parameter semigroup contained in $\P$ with $l(\Phi_1)=a,$ then it is easy to see that $l(\Phi_t)=a\cdot t.$ If $H$ is the generator of this semigroup, then we also define $l(H):=a.$

We will be interested in the following set of chordal generators.
\begin{definition}By $\K(\Ha, c)$ we denote the set of all infinitesimal generators $H$ of one-real parameter semigroups $\{\Phi_t\}_{t\geq0}$ contained in $\P$ with $l(H)\leq c.$
\end{definition}

\begin{remark}\label{rmk:1}
The set $\K(\Ha, c)$ can be characterized in various ways; see \cite{MR1201130}, Section 1 and \cite{MR1165862}, Proposition 2.2.

It is known that $H\in \K(\Ha, c)$ for some $c\geq 0$ if and only if $H$ maps $\Ha$ into $\overline{\Ha}$ and
\begin{equation}\label{limsup}\limsup_{y\to\infty}y |H(iy)|\leq c. \end{equation}
In fact, $l(H)=\limsup_{y\to\infty}y |H(iy)|.$\\
Furthermore, this is equivalent to: $H$ maps $\Ha$ into $\overline{\Ha}$ and
\begin{equation}\label{chordal_ineq} |H(z)|\leq \frac{c}{\Im(z)}
\end{equation}
for all $z\in \Ha$. The number $l(H)$ is the smallest constant such that this inequality holds.\\
Finally,  it is known that this property is equivalent to the fact that $-G$ is the Cauchy transform of a finite, non-negative Borel measure $\mu$ on $\R$, i.e. \begin{equation}\label{measure}H(z)=\int_{\R}\frac{\mu(du)}{u-z}.\end{equation}
The number $l(H)$ can be calculated by $l(H)=\mu(\R).$
\end{remark}

\begin{remark}\label{classP} It is easy to see that the following holds: if $f\in \P$ with $c=l(f)$, then $H:=f-id \in \K(\Ha,c)$ with $l(H)=c$.
\end{remark}

Let $C:\Ha \to \D,$ $C(z)=\frac{z-i}{z+i},$ be the Cayley map. We define $\K(\D, c)$ by 
$$\K(\D, c) = \{ C'(C^{-1}) \cdot (H \circ C^{-1}) \,|\, H \in \K(\Ha, c)\}.\footnote{If $\{\Phi_t\}_{t\geq0}$ is a semigroup in $\Ha$ with generator $H$, then $\{C \circ \Phi_t \circ C^{-1}\}_{t\geq 0}$ is a semigroup in $\D$ and its generator is given by $C'(C^{-1}) \cdot (H \circ C^{-1})$.}$$ 

The rest of this paper is organized as follows: In Section \ref{sec2} we look for an invariant characterization of chordal generators, i.e. of the sets $\K(\Ha, c)$ and $\K(\D, c),$ and we introduce the class $\K(\B_n, c)$ for the higher dimensional unit ball. It will turn out to be quite useful to study ``slices'' of this class, which is done in Section \ref{sec3}. In Section \ref{sec4} we introduce and study the class $\P_n,$ a higher dimensional analog of the class $\P$.

\section{Chordal generators in higher dimensions}\label{sec2}

\subsection{Invariant formulation for \texorpdfstring{$\K(\D, c)$}{} and \texorpdfstring{$\K(\Ha, c)$}{}}

For $R>0$ we let $E_\D(1, R)$ be the horodisc in $\D$ with center $1$ and radius $R,$ i.e. $$E_\D(1, R)=\left\{z\in\D \,|\, \frac1{|u_\D(z)|}<R\right\},$$  where $u_\D(z)=-\frac{1-|z|^2}{|1-z|^2}$ is the Poisson kernel in $\D$ with respect to $1$.\\
By using the Cayley map we define analogously $E_\Ha(\infty, R)= C^{-1}(E_\D(1,R))=\{z\in \Ha \,|\, \frac1{\Im(z)}<R\}.$
For $z\in \D$ and a tangent vector $v\in\C$ we denote by $|v|_{\D,z}$ the hyperbolic length of $v$ (with curvature -1), i.e. $$|v|_{\D,z}:=\frac{2|v|}{1-|z|^2}.$$ Furthermore, we let $R_\D(z)$ be the radius $R$ of the horodisc $E_\D(1,R)$ that satisfies $z\in \partial E(1,R)$; in short $R_\D(z)=\frac1{|u_\D(z)|}.$ Analogously,  for $z\in \Ha$ and $v\in\C,$ we define $R_\Ha(z):=1/\Im(z)$ and the hyperbolic length $|v|_{\Ha,z}:=|v|/\Im(z).$\\  

According to \eqref{chordal_ineq} we know that $H\in  \K(\Ha, c)$ if and only if $H$ maps $\Ha$ into $\overline{\Ha}$ and $|H(z)|\leq c/\Im(z)$ for all $z\in \Ha.$ By using the Berkson-Porta formula it is easy to see that we can rephrase this to: $H\in \K(\Ha, c)$ if and only if $H \in \III$ and $|H(z)|\leq c/\Im(z)$ for all $z\in \Ha$. \\
The last inequality is equivalent to $|H(z)|/\Im(z)\leq c/\Im(z)^2$ or

$$ |H(z)|_{\Ha,z} \leq \frac{c}{\Im(z)^2} = c\cdot R_\Ha(z)^2.$$

If we pass from $\Ha$ to $\D$ and transform $H$ into $G=C'(C^{-1}) \cdot (H \circ C^{-1}),$ then $G$ satisfies
$|G(C(z))|_{\D,C(z)} = |H(z)|_{\Ha,z}$ and we immediately get the following characterization. 

\begin{proposition}\label{observe} Let $G \in \II$. Then 
$$ G\in \K(\D, c) \quad \iff \quad |G(z)|_{\D,z} \leq c\cdot R_\D(z)^2 \quad \text{for all} \quad z\in\D.$$
 Let $H \in \III$. Then 
$$ H\in \K(\Ha, c) \quad \iff \quad |H(z)|_{\Ha,z} \leq c\cdot R_\Ha(z)^2 \quad \text{for all} \quad z\in\Ha.$$
\end{proposition}

\subsection{Chordal generators in the unit ball}

For $n\in\N$, let $u_{n}$ be the pluricomplex Poisson kernel in $\B_n$ with pole at $e_1:=(1,0,...,0),$ i.e.
$$u_{\B_n,p}=-\frac{1-\|z\|^2}{|1-z_1|^2}.$$
The level sets of $u_{\B_n}$ are exactly the boundaries of horospheres with center $e_1,$ more precisely, the set 
$$E_{\B_n}(e_1,R):=\{z\in \B_n \,|\, |u_{\B_n}(z)|^{-1}<R\}, R>0,$$
is the horosphere with center $e_1$ and radius $R.$\\
Furthermore, for $z\in \B_n$ and $v\in \C^n$ we denote by $\|v\|_{\B_n,z}$ the Kobayashi-hyperbolic length of the vector $v$ with respect to $z.$

Motivated by Proposition \ref{observe}, we make the following definition. 

\begin{definition}
Let $c\geq0$. We define the class $\K(\B_n, c)$ to be the set of all infinitesimal generators $G$ on $\B_n$ such that for all $z\in \B_n$ the following inequality holds:  \begin{equation}\label{bas} \|G(z)\|_{\B_n,z} \leq \frac{c}{u_{\B_n}(z)^2}. \end{equation}
\end{definition}

\begin{remark} $\K(\B_n, c)$ is a compact family: Montel's theorem and the definition of $\K(\B_n, c)$ immediately imply that it is a normal family. If a sequence $(G_n) \subset \K(\B_n, c)$ converges locally uniformly to $G: \B_n \to \C^n,$ then $G$ is holomorphic and also an infinitesimal generator which can be seen by using the characterization given in \cite{MR2578602}, Theorem 0.2. Of course, $G$ also satisfies \eqref{bas} and we conclude $G \in \K(\B_n, c).$
\end{remark}

Just as we passed from $\D$ to $\Ha$ in one dimension, we can pass from the unit ball $\B_n$ to the Siegel upper half-space $\Ha_n=\{(z_1,\tilde{z})\in\C^n \,|\, \Im(z_1)>\|\tilde{z}\|^2\}$
 in order to get simpler formulas:\\
The Cayley map $$C:\Ha_n \to \B_n, \quad C(z)=(C_1(z),...,C_n(z))=\left(\frac{z_1-i}{z_1+i},\frac{2z_2}{z_1+i},...,\frac{2z_n}{z_1+i}\right),$$
maps $\Ha_n$ biholomorphically onto $\B_n$. It extends to a homeomorphism from the one-point compactification $\widehat{\Ha}_n=\Ha_n \cup \partial \Ha_n\cup \{\infty\}$ of $\Ha_n \cup \partial \Ha_n$ to the closure of $\B^n$.\\
The pluricomplex Poisson kernel transforms as follows:
$$u_{\Ha_n}(z):=u_{\B_n}(C(z))=-\Im(z_1)+\|\tilde{z}\|^2.$$
Thus, we define the horosphere $E_{\Ha_n}(\infty, R)$ with center $\infty$ and radius $R>0$  by
$$E_{\Ha_n}(\infty, R):=\{z\in \Ha_n \,|\,\Im(z_1)-\|\tilde{z}\|^2>\frac1{R}\}.$$
For $v\in\C^n$ and $z\in \Ha_n$ we let $\|v\|_{\Ha_n, z}$ be the Kobayashi hyperbolic length of $v.$\\

Let $c\geq0$. We define the class $\K(\Ha_n, c)$ to be the set of all infinitesimal generators $H$ on $\Ha_n$ satisfying the inequality $$ \|H(z)\|_{\Ha_n,z} \leq \frac{c}{u_{\Ha_n}(z)^2} $$
for all $z\in \Ha_n.$ Then we have
$$\K(\B_n, c) = \{ C'(C^{-1}) \cdot (H \circ C^{-1}) \,|\, H \in \K(\Ha_n, c)\}.$$

From now on we will stay in the upper half-space $\Ha_n,$ where most of the computations we need take a simpler form.

\section{Slices}\label{sec3}

\subsection{Normalized geodesics and slices}
For any $H\in \text{Inf}(\Ha_n)$ one can consider one-dimensional slices by using the so called \emph{Lempert projection devices}; see \cite{MR3130328}, Section 3.\\
If $w\in \Ha_n,$ then there exists a unique complex  passing through $w$ and $\infty.$ Let us choose a parametrization 
$\varphi:\Ha \to \Ha_n$ of this geodesic. There exists a unique holomorphic map $P:\Ha_n\to \Ha_n$ with $P^2=P$ and $P\circ \varphi = \varphi.$ Define $\tilde{P}=\varphi^{-1}\circ P.$ Then  $$h_\varphi:\Ha\to\C, \qquad h_\varphi(\zeta)=d\tilde{P}(\varphi(\zeta))\cdot H(\varphi(\zeta)),$$

is an infinitesimal generator on $\Ha;$ see \cite{MR3130328}, p. 6. \\
We will need special parametrizations of these geodesics:
In \cite{MR2181760}, p. 516, it is shown that for any complex geodesic $\varphi:\Ha\to \Ha_n$ with $\varphi(\infty)=\infty,$ there exists $a_\varphi>0$ such that $$ u_{\Ha_n}(\varphi(\zeta))=a_{\varphi}\cdot u_{\Ha}(\zeta) $$
for all $\zeta \in \Ha.$ Call a geodesic $\varphi:\Ha\to \Ha_n$ \emph{normalized} if $\varphi(\infty)=\infty$ and $a_\varphi = 1.$

\begin{lemma} Let $a \in \C$ and $\gamma\in \C^{n-1}$ such that $(a, \gamma)\in \Ha_n.$ Then the map 
$$\label{geodesic}\varphi_\gamma:\Ha\to\Ha_n, \quad \varphi_{\gamma}(\zeta):=(\zeta+i \|\gamma\|^2, \gamma),$$
 is a normalized geodesic through $(a, \gamma).$ Furthermore, if $H=(H_1, \tilde{H})\in \text{Inf}(\Ha_n)$, then the slice $h_\gamma:=h_{\varphi_\gamma}$ of $H$ with respect to $\varphi_\gamma$ is given by 
\begin{equation}\label{slice} h_\gamma(\zeta) = H_1(\varphi_\gamma(\zeta))-2i\overline{\gamma}^T \cdot \tilde{H}(\varphi_\gamma(\zeta)).
\end{equation}
\end{lemma}
\begin{proof}
Let $\psi:\D\to\B_n$ be a complex geodesic with $\psi(1)=e_1.$ As a parametrization for $\psi$ one can choose (see Section 3 in \cite{MR3130328})
$\psi(\zeta) = (\alpha^2(\zeta-1)+1, \alpha(\zeta-1)\beta),$
where $\alpha>0$ and $\beta\in \C^{n-1}$ such that $\|\beta\|^2=1-\alpha^2.$ Then $C^{-1}(\psi(\zeta))=(i\frac{2+\alpha^2(\zeta-1)}{\alpha^2(1-\zeta)}, i\beta/\alpha)$ and 
$$\zeta \mapsto C^{-1}(\psi(C_1(\zeta)))=(-i+\frac{\zeta+i}{\alpha^2}, i\beta/\alpha)=(\frac{\zeta}{\alpha^2}+i
\frac{1-\alpha^2}{\alpha^2}, i\beta/\alpha)=(\frac{\zeta}{\alpha^2}+i
\left\|\frac{\beta}{\alpha}\right\|^2, i\beta/\alpha)$$
is a complex geodesic from $\Ha$ to $\Ha_n$. A reparametrization [$ \zeta/\alpha^2$  to $\zeta$] and setting $\gamma=i\beta/\alpha$ gives
 the geodesic
\begin{equation}\varphi_\gamma(\zeta)=(\zeta+i\|\gamma\|^2,\gamma).\end{equation}

This complex geodesic is normalized because it satisfies $\varphi_\gamma(\infty) = \infty$ and 
$$ u_{\Ha_n}(\varphi_\gamma(\zeta)) = \Im(\zeta + i\|\gamma\|^2) - \|\gamma\|^2 = \Im(\zeta)  = u_{\Ha}(\zeta).$$

The projection onto $\varphi_\gamma(\Ha)$ is given by

\begin{equation}\label{proexplicit}
P(z_1, \tilde{z})=(z_1-2i\overline{\gamma}^T \cdot \tilde{z}+2i\|\gamma\|^2, \gamma).
\end{equation}

Clearly, $P$ is holomorphic and maps $\Ha_n$ onto $\varphi_\gamma(\Ha)$ because
\begin{eqnarray*}
\Im(z_1-2i\overline{\gamma}^T \cdot \tilde{z}+2i\|\gamma\|^2)&=&\Im(z_1)-2\Im(i\overline{\gamma}^T \cdot \tilde{z})+2\|\gamma\|^2\\
&\geq&
\|\tilde{z}\|^2-2\|\gamma\|\|\tilde{z}\|+\|\gamma\|^2+\|\gamma\|^2=(\|\gamma\|-\|\tilde{z}\|)^2+\|\gamma\|^2\geq \|\gamma\|^2.
\end{eqnarray*}
Furthermore, 
$$(P\circ P)(z_1,\tilde{z})=(z_1-2i\overline{\gamma}^T \tilde{z}+2i\|\gamma\|^2-2i\overline{\gamma}^T \gamma+2i\|\gamma\|^2,\gamma)=
(z_1-2i\overline{\gamma}^T\tilde{z}+2i\|\gamma\|^2,\gamma)=P(z_1,\tilde{z}).$$
Thus, the inverse $\tilde{P}:\Ha_2\to \Ha, \tilde{P}=\varphi_\gamma^{-1}\circ P,$ is given by $\tilde{P}(z_1,\tilde{z})=(z_1-2i\overline{\gamma}^T\tilde{z}+i\|\gamma\|^2).$\\
If $H(z)=(H_1(z),\tilde{H}(z))$ is a generator on $\Ha_n,$ we get the slice reduction
$$h_{\varphi_\gamma}(\zeta)=d\tilde{P}(\varphi_\gamma(\zeta))\cdot H(\varphi_\gamma(\zeta))=H_1(\varphi_\gamma(\zeta))-2i\overline{\gamma}^T\cdot \tilde{H}(\varphi_\gamma(\zeta)).$$
\end{proof}

\subsection{Some explicit formulas}
Later on we will need explicit formulas of the Kobayashi norms of $dP(z)H(z)$ and $H(z)-dP(z)\cdot H(z).$ The following  lemma is proven in the Appendix.

\begin{lemma}\label{formulas0}
Let $a\in \C, p,v\in \C^{n-1}$ and $z=(z_1, \tilde{z})\in \Ha_n.$ Then the following formulas hold:
\begin{equation}\label{normproj00}\left\|\binom{a}{0}\right\|_{\Ha_n, z} = \frac{|a|}{|u_{\Ha_n}(z)|},
\end{equation}
\begin{equation}\label{normorth00}
\left\|\binom{2i \overline{p}^T v}{v}\right\|_{\Ha_n, z}= 2\frac{\sqrt{ \|v\|^2|u_{\Ha_n}(z)|+|\overline{(p-\tilde{z})}^T v|^2}}{|u_{\Ha_n}(z)|},
\end{equation}
\begin{equation}\label{orth00} \left\|\binom{a-2i \overline{\tilde{z}}^T v}{0} + \binom{2i \overline{\tilde{z}}^T v}{v}\right\|_{\Ha_n, z}^2=
 \left\|\binom{a-2i \overline{\tilde{z}}^T v}{0}\right\|_{\Ha_n, z}^2+
 \left\|\binom{2i \overline{\tilde{z}}^T v}{v}\right\|_{\Ha_n, z}^2 .\end{equation}
\end{lemma}

By using Lemma \ref{formulas0} we obtain the following explicit expressions.

\begin{lemma}\label{formulas} Let $H=(H_1, \tilde{H})\in \text{Inf}(\Ha_n)$ and fix $z\in \Ha_n$. Denote by $P$ the projection onto the complex geodesic through $z$ and $\infty$. Then the following formulas hold:
\begin{equation}\label{eq:7} dP(z)\cdot H(z) = (H_1(z)-2i\overline{\tilde{z}}^T\tilde{H}(z), 0), \quad \quad 
H(z)-dP(z)\cdot H(z) = (2i\overline{\tilde{z}}^T\tilde{H}(z), \tilde{H}(z)).
\end{equation}
Furthermore,
\begin{equation}\label{orthogonal} \|H(z)\|_{\Ha_n, z}^2=\|dP(z)\cdot H(z)\|_{\Ha_n, z}^2 + \|H(z)-dP(z)\cdot H(z)\|_{\Ha_n, z}^2, \end{equation}
\begin{equation}\label{normproj}\|dP(z)H(z)\|_{\Ha_n, z} = \frac{|H_1(z)-2i\overline{\tilde{z}}^T\tilde{H}(z)|}{|u_{\Ha_n}(z)|},\end{equation}
\begin{equation}\label{normorth}\|H(z)-dP(z)\cdot H(z)\|_{\Ha_n, z} = 2\frac{\|\tilde{H}(z)\|}{\sqrt{|u_{\Ha_n}(z)|}}.\end{equation}
\end{lemma}
\begin{proof}
The formulas for $dP(z)H(z)$ and $H(z)-dP(z)H(z)$ follow from the explicit form \eqref{proexplicit}.\\
Equation \eqref{orthogonal} follows from \eqref{orth00} with $a=H_1(z)$ and
$v=\tilde{H}(z).$\\
Furthermore, equation \eqref{normproj} follows directly from \eqref{normproj00} with $a=H_1(z)-2i\overline{\tilde{z}}^T\tilde{H}(z)$ and equation \eqref{normorth} from \eqref{normorth00} by setting $p=\tilde{z}$ and $v=\tilde{H}.$
\end{proof}

\subsection{Slices of generators in \texorpdfstring{$\K(\Ha_n, c)$}{} and examples}

\begin{proposition}\label{subset} Let $c\geq 0$ and $H \in \K(\Ha_n, c).$ Then every normalized slice $h_\gamma$ of $H$ belongs to $\K(\Ha, c).$
\end{proposition}
\begin{proof}
Fix $\gamma\in \C^{n-1}$ and $\zeta \in \Ha$ and let $z=\varphi_\gamma(\zeta).$\\
Furthermore, let $P$ be the projection onto $\varphi_\gamma(\Ha)$. Now we write $H(z)$ as $$H(z)=dP(z)\cdot H(z) + (H(z)-dP(z)H(z)).$$
As $H\in \K(\Ha_n, c)$, equation \eqref{orthogonal} implies
$$ \|H(z)\|_{\Ha_n, z}^2=\|dP(z)\cdot H(z)\|_{\Ha_n, z}^2 + \|H(z)-dP(z)H(z)\|_{\Ha_n, z}^2 \leq   \frac{c^2}{ u_{\Ha_n}(z)^4}.$$
In particular,
\begin{equation}\label{eq1}  \|dP(z)\cdot H(z)\|_{\Ha_n, z} \leq \frac{c}{u_{\Ha_n}(z)^2}.\end{equation}

By the definition of the slice $h_\gamma$ we have
$$ dP(\varphi_\gamma(\zeta)) \cdot H(\varphi_\gamma(\zeta)) = (d\varphi_\gamma)(\zeta)\cdot h_\gamma(\zeta)$$
and consequently $$\|dP(\varphi_\gamma(\zeta)) \cdot H(\varphi_\gamma(z))\|_{\Ha_n, \varphi_\gamma(\zeta)} = \|(d\varphi_\gamma)(\zeta)\cdot h_\gamma(\zeta)\|_{\Ha_n, \varphi_\gamma(\zeta)}=|h_\gamma(\zeta)|_{\Ha, \zeta}.$$ The last equality holds as $\varphi_\gamma$ is a complex geodesic. Equation \eqref{eq1} implies $$ |h_\gamma(\zeta)|_{\Ha, \zeta} \leq  \frac{c}{u_{\Ha_n}(\varphi_\gamma(\zeta))^2}=\frac{c}{u_{\Ha}(\zeta)^2},$$
where the last equality holds as $\varphi_\gamma$ is normalized. Hence, $h_\gamma \in \K(\Ha, c).$ 
\end{proof}

\begin{remark}\label{rem2} If two holomorphic functions $H_1, H_2:\Ha_n\to\C^n$ have the same slices, i.e. $dP(z)H_1(z)=dP(z)H_2(z)$ for all $z\in \Ha_n,$ then $H_1=H_2$; see the proof of Theorem 3.2 in \cite{MR2811960}.
\end{remark}

\begin{example}
The family $\{\Phi_t(z)=(z_1, e^{-it/z_1}z_2)\}_{t\geq0}$ is a semigroup on $\Ha_2.$ Its generator $H$ is given by 
 
$$H(z_1,z_2)=(0, -i\frac{z_2}{z_1}).$$ Thus, for $\gamma\in\C$ the slice $h_\gamma$ has the form 
$$h_\gamma(z)=-2i\overline{\gamma}\cdot -i\frac{\gamma}{z+i|\gamma|^2}=
\frac{-2|\gamma|^2}{z+i|\gamma|^2}.$$
Consequently, the limit $\lim_{y\to\infty} y\cdot |h(iy)|= 2 |\gamma|^2$ exists, but does not have an upper bound that is independent of $\gamma.$  Proposition \ref{subset} implies that for any $c\geq 0,$ $H\not\in \K(\Ha_2, c)$. \hfill $\bigstar$
\end{example}

\begin{example}
Let
$$H:\Ha_2\to\C^2, \qquad H(z_1, z_2) = \binom{ \frac{-1}{z_1}} {\frac{z_2}{2 z_1^2}}.$$
For $\gamma \in \C$ the slice $h_\gamma$ is given by
$$ h_\gamma(\zeta) =  \frac{-1}{\zeta +i|\gamma|^2}  - 2i\overline{\gamma} \cdot \frac{\gamma}{2(\zeta + i|\gamma|^2)^2}=
\frac{-\zeta -2i|\gamma|^2}{(\zeta +i|\gamma|^2)^2}=\frac{(-\zeta -2i|\gamma|^2)(\overline{\zeta}^2 -2i|\gamma|^2\overline{\zeta} -|\gamma|^4 )}{|\zeta +i|\gamma|^2|^4}. $$
Let us write $\zeta=x+iy;$ $x\in\R ,y\in (0,\infty).$ Then a small calculation gives
$$ \Im(h_\gamma(\zeta)) =\frac{y(x^2+y^2) +4y^2|\gamma|^2  +5y|\gamma|^4+2|\gamma|^6}{|\zeta +i|\gamma|^2|^4}>0.$$
 Furthermore, $$\limsup_{y\to \infty} y |h_{\gamma}(iy)|=1.$$
Hence, $h_\gamma \in \K(\Ha, 1).$ So each slice is an infinitesimal generator in $\Ha$ and by \cite{MR3130328}, Proposition 3.8, the function $H$ is an infinitesimal generator in $\Ha_2.$\\ 
Now let $(z_1, z_2)\in \Ha_2$ and write $z_1=x+iy,$ $x,y\in\R.$ Then we get (an explicit formula of the Kobayashi
metric is given in the appendix)
$$  u_{\Ha_2}(z)^4 \cdot \|H(z)\|_{\Ha_2, z}^2 = (y-|z_2|^2)^2  \cdot \frac {x^2+y^2+3 |z_2|^2 y}{(x^2+y^2)^2}\underset{y\geq |z_2|^2}{\leq}
y^2 \cdot \frac {x^2+y^2+3 y^2}{(x^2+y^2)^2}\leq\frac {x^2+4y^2}{x^2+y^2}\leq 4.$$
Consequently, $H \in \K(\Ha_2, 2).$ \hfill $\bigstar$
\end{example}

\begin{question} Let $H:\Ha_n \to \C^n$ be an infinitesimal generator. Assume there exists $c\geq 0$ such that 
$h_\gamma \in \K(\Ha, c)$ for every $\gamma \in \C^{n-1}.$ Does this imply that $H\in \K(\Ha_n, C)$ for some $C\geq c$?
\end{question}

 \section{Univalent functions with hydrodynamic normalization}\label{sec4}

Motivated by Remark \ref{classP} we define the following generalization of the class $\P,$ where $id$ stands for the identity mapping on $\Ha_n.$

\begin{definition}
$$\P_n:=\{f:\Ha_n\to\Ha_n \,|\, \text{$f$ is univalent and $f-id\in \K(\Ha_n, c)$ for some $c\geq 0$}\}.$$
\end{definition}

\begin{remark} \label{autoinf} It is important to note that if $f:\Ha_n \to \Ha_n$ is a holomorphic self-mapping, then the map $f-id$ is automatically an infinitesimal generator; see \cite{MR2022955}, p. 207.    
\end{remark}

\subsection{Basic properties of \texorpdfstring{$\P_n$}{}}

The following proposition summarizes some basic properties of $\P_n.$

\begin{proposition}\label{basics}${}$
\begin{itemize} 
\item[a)] $\P_n$ contains no automorphism of $\Ha_n$ except the identity.
 		\item[b)] Let $\alpha:\Ha_n \to \Ha_n$ be an automorphism of $\Ha_n$ with $\alpha(\infty)=\infty.$ If $f\in \P_n$, then $\alpha^{-1}\circ f\circ \alpha\in \P_n$.
	\item[c)] Let $f\in\P_n.$ Then $f(E_{\Ha_n}(\infty, R)) \subset E_{\Ha_n}(\infty, R)$ for every $R>0.$
	\item[d)] Let $f\in \P_n$ and write $f(z)=z+H(z)$ with $H=(H_1, \tilde{H})\in \K(\Ha_n,c).$ Then \begin{equation}\label{eq3}
\text{$\|\tilde{H}(z)\|^2 \leq |H_1(z)-2i\overline{\tilde{z}}^T \tilde{H}|$ for all $z=(z_1, \tilde{z})\in \Ha_n.$}
\end{equation}
	\item[e)] Let $f\in\P_n.$ Then there exists $R>0$ such that $E_{\Ha_n}(\infty, R)\subset f(\Ha_n).$
\end{itemize}
\end{proposition}
\begin{proof}
The statements a) and b) can easily be shown by using the explicit form of automorphisms of $\Ha_n;$ see Proposition 2.2.4 in \cite{MR1098711}.\\
The   statement c) is just Julia's lemma: Write $f(z)=z+H(z)$ and let us pass to the unit ball and define $\tilde{f}:\B_n\to \B_n, \tilde{f} = C\circ f\circ C^{-1}.$ Then  

$$ \tilde{f}=\frac1{2i+H_1(C^{-1}(z))-z_1H_1(C^{-1}(z))}\left[\binom{(1-z_1)H_1(C^{-1}(z))}{2(1-z_1)\tilde{H}(C^{-1}(z))}+2i z\right]. $$ 
By taking the sequence $z_n=(1-1/n, 0)$ it is easy to see that

$$\text{$\lim_{n\to \infty}\tilde{f}(z_n)=e_1$ \quad and \quad $\lim_{n\to \infty} \frac{1-\|\tilde{f}(z_n)\|}{1-\|z_n\|}=1$,}$$
i.e. $e_1$ is a boundary regular fixed point of $\tilde{f}$ with boundary dilatation coefficient $\leq 1.$ Julia's lemma (see Theorem
2.2.21 in \cite{MR1098711}) implies that $\tilde{f}(E_{\B_n}(e_1, R)) \subset E_{\B_n}(e_1, R)$ for any $R>0.$\\
Inequality d) follows directly from c): 
Let $z=(z_1, \tilde{z})\in \Ha_n$. Another formulation of c) is $-u_{\Ha_n}(z+H(z))\geq -u_{\Ha_n}(z),$ or more explicitly
\begin{eqnarray*}
&&\Im(z_1)+\Im(H_1(z))-\|\tilde{z}+\tilde{H}(z)\|^2 \geq \Im(z_1)-\|\tilde{z}\|^2 \\
&\iff& \Im(H_1(z))\geq\|\tilde{z}+\tilde{H}(z)\|^2 -\|\tilde{z}\|^2 = 2\Re(\overline{\tilde{z}}^T \tilde{H}(z)) + \|\tilde{H}(z)\|^2 \\
&\iff& \Im(H_1(z) - 2i \overline{\tilde{z}}^T \tilde{H}(z)) \geq \|\tilde{H}(z)\|^2.
 \end{eqnarray*}

From this inequality it follows that $\|\tilde{H}(z)\|^2 \leq |H_1(z)-2i\overline{\tilde{z}}^T \tilde{H}|$ for all $z\in \Ha_n.$\\

Finally we prove e):\\
Let $f\in \P_n$ and write $f(z)=z+H(z)$ with $H\in \K(\Ha_n,c).$ Because of c) $f$ maps the horosphere $E_{\Ha_n}(\infty,1)$ into itself. Hence the statement is proven if we can show that $u_{\Ha_n}$ is bounded on $f(\partial E_{\Ha_n}(\infty, 1))$.\\ 
Let $z\in \Ha_n$ with $z\in \partial E_{\Ha_n}(\infty, 1)$, i.e. $|u_{\Ha_n}(z)|=1.$ Furthermore, we choose $\zeta\in \Ha$ and $\gamma\in\C$ such that $\varphi_\gamma(\zeta)=z.$ Note that this implies $|u_{\Ha}(\zeta)|=\Im(\zeta)=1.$ \\
 Let $P$ be the projection onto $\varphi_\gamma(\Ha).$\\

 Then we have $|u_{\Ha_n}(f(z))|=|u_{\Ha_n}(z+H(z))|=|u_{\Ha_n}(\underbrace{z+dP(z)H(z)}_{=:w}+\underbrace{H(z)-dP(z)H(z)}_{=:v})|.$
As $dP(z) \cdot dP(z) = dP(z),$ we have $dP(z) \cdot v =0.$ A small calculation (see also Lemma 3.1 in  \cite{MR2811960}) gives
 $v\in T^\C_z \partial E_{\Ha_n}(\infty, 1).$ Furthermore, also $w \in \varphi_\gamma(\Ha)$ and $dP(z)=dP(w)$ and we get $v\in T^\C_w \partial E_{\Ha_n}(\infty, |u_{\Ha_n}(w)|^{-1}).$ As $E_{\Ha_n}(\infty, |u_{\Ha_n}(w)|^{-1})=\{z\in\Ha_n \,|\, |u_{\Ha_n}(z)|>|u_{\Ha_n}(w)|\}$ is convex this implies
\begin{eqnarray*}&&|u_{\Ha_n}(w+v)| \leq |u_{\Ha_n}(w)|  = |u_{\Ha_n}(z+dP(z)H(z))| \underset{\text{Lemma \ref{formulas}}}{=} |u_{\Ha_n}(z+(h_\gamma(\zeta),0))| \\
&=& \Im(z_1)-\|\tilde{z}\|^2+\Im(h_\gamma(\zeta)) \leq \Im(z_1)-\|\tilde{z}\|^2+|h_\gamma(\zeta)|\\
&=& |u_{\Ha_n}(z)|+|h_\gamma(\zeta)|=1+|h_\gamma(\zeta)|\leq 1+\frac{c}{\Im(\zeta)}=1+c.\end{eqnarray*}

Consequently, $f(\Ha_n) \supset f(E_{\Ha_n}(\infty,1))\supset E_{\Ha_n}(\infty,1+c).$
\end{proof}

\begin{theorem}\label{semigroup} $\P_n$ is a semigroup: If $f,g\in\P_n,$ then $f\circ g \in \P_n.$
\end{theorem}
\begin{proof} Let $f,g\in \P_n$ with $F=(F_1, \tilde{F}):=f-id, G=(G_1, \tilde{G}):=g-id$ and
$$\|F(z)\|_{\Ha_n,z}\leq \frac{c}{u_{\Ha_n}(z)^2}, \quad \quad \|G(z)\|_{\Ha_n,z}\leq \frac{d}{u_{\Ha_n}(z)^2}$$
for all $z\in\Ha_n$. Let $z=(z_1, \tilde{z})\in \Ha_n$ and $p=(p_1, \tilde{p}):=z+G(z)$.\\
From Remark \ref{autoinf} we know that $f\circ g-id$ is an infinitesimal generator on $\Ha_n.$ It remains to estimate the hyperbolic metric of this generator. We have 
\begin{eqnarray*}&&\|(f\circ g)(z)-z\|_{\Ha_n,z} = \|G(z)+ F(z+G(z))\|_{\Ha_n,z} \\
&\leq& \|G(z)\|_{\Ha_n,z} + \|F(z+G(z))\|_{\Ha_n,z}  \leq \frac{d}{u_{\Ha_n}(z)^2} + \|F(p)\|_{\Ha_n,z} \\
&\leq& \frac{d}{u_{\Ha_n}(z)^2} +
\|(F_1(p)-2i \overline{\tilde{p}}^T \tilde{F}(p), 0)\|_{\Ha_n, z} + \|(2i \overline{\tilde{p}}^T \tilde{F}(p), \tilde{F}(p))\|_{\Ha_n, z}.
\end{eqnarray*}
Note that $F_1(p)-2i \overline{\tilde{p}}^T \tilde{F}(p)$ corresponds to the slice of $F$ with respect to the geodesic through $p$ and infinity. Because of Proposition \ref{subset} we know that
$$|F_1(p)-2i \overline{\tilde{p}}^T \tilde{F}(p)|\leq \frac{c}{ |u_{\Ha_n}(p)|} \leq \frac{c}{ |u_{\Ha_n}(z)|},$$
where the second inequality follows from Proposition \ref{basics} c). Together with equation \eqref{normproj00}, this implies
\begin{equation}\label{eq5}
\|(F_1(p)-2i \overline{\tilde{p}}^T \tilde{F}(p), 0)\|_{\Ha_n, z}=\frac{|(F_1(p)-2i \overline{\tilde{p}}^T \tilde{F}(p)|}{|u_{\Ha_n}(z)|}\leq  \frac{c}{u_{\Ha_n}(z)^2}.
 \end{equation}
It remains to show that there exists a constant $C>0$ such that 
$$\|(2i \overline{\tilde{p}}^T \tilde{F}(p), \tilde{F}(p))\|_{\Ha_n, z} \leq \frac{C}{u_{\Ha_n}(z)^2}.$$
First, equation \eqref{normorth00} gives
\begin{align}\label{eq:8}\begin{split}
&\|(2i \overline{\tilde{p}}^T \tilde{F}(p), \tilde{F}(p))\|_{\Ha_n, z}= 2\frac{\sqrt{ \|\tilde{F}(p)\|^2|u_{\Ha_n}(z)|+|\overline{(\tilde{p}-\tilde{z})}^T\tilde{F}(p)|^2}}{|u_{\Ha_n}(z)|} \\
&\leq  2\frac{\sqrt{ \|\tilde{F}(p)\|^2|u_{\Ha_n}(z)|+\|(\tilde{p}-\tilde{z})\|^2\cdot \|\tilde{F}(p)\|^2}}{|u_{\Ha_n}(z)|} 
= 2\frac{\|\tilde{F}(p)\|}{|u_{\Ha_n}(z)|} \sqrt{|u_{\Ha_n}(z)|+\|\tilde{G}(z)\|^2}.
\end{split}\end{align}
Now we differentiate between two cases. Case 1: $|u_{\Ha_n}(z)|\geq 1.$ \\
The equations \eqref{orthogonal} and \eqref{normorth} imply $\frac{2\|\tilde{F}(p)\|}{\sqrt{|u_{\Ha_n(p)}|}}\leq \|\tilde{F}(p)\|_{\Ha_n, p}\leq \frac{c}{u_{\Ha_n}(p)^2},$ thus \begin{equation}\label{eq:9}
\|\tilde{F}(p)\|\leq  \frac{c}{2|u_{\Ha_n}(p)|^{3/2}} \leq \frac{c}{2|u_{\Ha_n}(z)|^{3/2}}.\end{equation}
 In the same way we get
 \begin{equation}\label{eq:10}\|\tilde{G}(z)\|\leq \frac{d}{2|u_{\Ha_n}(z)|^{3/2}}.\end{equation} 
Combining \eqref{eq:9} with \eqref{eq:8} gives
\begin{eqnarray*}
&&\|(2i \overline{\tilde{p}}^T \tilde{F}(p), \tilde{F}(p))\|_{\Ha_n, z} \leq  \frac{c}{|u_{\Ha_n}(z)||u_{\Ha_n}(z)|^{3/2}} \sqrt{|u_{\Ha_n}(z)|+\|\tilde{G}(z)\|^2} \\
&=&
 \frac{c}{|u_{\Ha_n}(z)|^{2}} \sqrt{1+\frac{\|\tilde{G}(z)\|^2}{|u_{\Ha_n}(z)|}} 
\underset{\eqref{eq:10}}{\leq}
 \frac{c}{|u_{\Ha_n}(z)|^{2}} \sqrt{1+\frac{d^2}{4|u_{\Ha_n}(z)|^{4}}} \leq \frac{c\sqrt{1+\frac{d^2}{4}}}{|u_{\Ha_n}(z)|^{2}}. 
\end{eqnarray*}

Case 2: $|u_{\Ha_n}(z)|\leq 1.$\\
From equation \eqref{eq5} we know that 
$|F_1(p)-2i\overline{\tilde{p}}^T \tilde{F}(p)|\leq \frac{c}{|u_{\Ha_n}(z)|}$ and equation \eqref{eq3} implies
$$ \|\tilde{F}(p)\| \leq  \frac{\sqrt{c}}{\sqrt{|u_{\Ha_n}(z)|}}.  $$
Similarly we get $$ \|\tilde{G}(z)\| \leq  \frac{\sqrt{d}}{\sqrt{|u_{\Ha_n}(z)|}}. $$
Hence, we obtain with \eqref{eq:8}:
\begin{eqnarray*}
&&\|(2i \overline{\tilde{p}}^T \tilde{F}(p), \tilde{F}(p))\|_{\Ha_n, z} \leq  2\frac{\sqrt{c}}{|u_{\Ha_n}(z)|^{3/2}} \sqrt{|u_{\Ha_n}(z)|+\|\tilde{G}(z)\|^2} \\
&\leq&
  2\frac{\sqrt{c}}{|u_{\Ha_n}(z)|^{3/2}} \sqrt{|u_{\Ha_n}(z)|+\frac{d}{|u_{\Ha_n}(z)|}} =
2\frac{\sqrt{c}}{|u_{\Ha_n}(z)|^{2}} \sqrt{u_{\Ha_n}(z)^2+d} \\
&\leq &
2\frac{\sqrt{c}}{|u_{\Ha_n}(z)|^{2}} \sqrt{1+d}.
\end{eqnarray*}

\end{proof}

\subsection{Semigroups with generators in \texorpdfstring{$\K(\Ha_n, c)$}{}}

\begin{theorem}\label{thm:1} Let $\{\Phi_t\}_{t\geq 0}$ be a semigroup on $\Ha_n$ with generator $H\in \K(\Ha_n, c).$ Then $\Phi_t\in \P_n$ for every $t\geq0.$
\end{theorem}
\begin{proof}
Firstly, for every $t\geq 0$ and $R>0,$ the map $\Phi_t$ maps the horosphere $E_{\Ha_n}(\infty, R)$ into itself, i.e.
\begin{equation}\label{eq:12} |u_{\Ha_n}(\Phi_t(z))|\geq |u_{\Ha_n}(z)|\end{equation}
for every $z\in \Ha_n.$ This can be seen as follows:\\
Let $G$ be the corresponding generator in the unit ball, i.e. $G=C'(C^{-1}) \cdot (H \circ C^{-1}).$ Then $G$ satisfies the inequality
$$ \|G(z)\|\leq\|G(z)\|_{\B_n, z} \leq \frac{c}{u_{\B_n}(z)^2}=\frac{c|1-z_1|^4}{(1-\|z\|^2)^2}. $$
Putting $z=r\cdot e_1$ gives 
$$ \|G(re_1)\| \leq \frac{c(1-r)^4}{(1-r^2)^{2}}=\frac{c(1-r)^{2}}{(1+r)^{2}}. $$
From this it follows immediately that $$\lim_{(0,1)\ni r\to 1}G(re_1)=0 \quad \text{and} \quad  \lim_{(0,1)\ni r\to 1}
\frac{G_1(re_1)}{r-1}=0.$$
Theorem 0.3 in \cite{MR2578602} implies that $e_1$ is a boundary regular fixed point for the generated semigroup with boundary dilatation coefficient 1. Hence we can apply Julia's lemma.\\

Let $z=(z_1,z_2)\in \Ha_n$ and write $\Phi_t=(\Phi_{1,t}, \tilde{\Phi}_t),$ $H=(H_1, \tilde{H}).$ The semigroup $\Phi_t$ satisfies the integral equation  
$$ \Phi_t(z) = z + \int_0^{t}H(\Phi_s(z)) \, ds. $$
Similarly to the proof of Theorem \ref{semigroup}, equation \eqref{eq:9}, we deduce from the fact that $H\in\K(\Ha_n,c)$ and equations \eqref{orthogonal} and \eqref{normorth} that 
\begin{equation}\label{eq:11}
\|\tilde{H}(\Phi_t(z))\| \leq \frac{c}{2|u_{\Ha_n}(z)|^{3/2}}\end{equation}
for every $z\in \Ha_n$ and $t\geq 0;$ and similarly to equation \eqref{eq5} we deduce that
\begin{equation}\label{eq:14}
\|(H_1(\Phi_t(z))-2i \overline{\tilde{\Phi}_t}^T \tilde{H}(\Phi_t(z)), 0)\|_{\Ha_n, z}\leq  \frac{c}{ u_{\Ha_n}(z)^2}
 \end{equation}
for every $z\in \Ha_n$ and $t\geq 0.$\\
First we get
\begin{equation}\label{ineq}\|\tilde{\Phi}_t-\tilde{z}\|\leq\int_0^s \|\tilde{H}(\Phi_\tau(z))\| \, d\tau \leq
\int_0^s \frac{c}{2|u_{\Ha_n}(z)|^{3/2}} \, d\tau =  \frac{cs}{2|u_{\Ha_n}(z)|^{3/2}}. \end{equation}
Case 1: $|u_{\Ha_n}(z)|\geq 1$. Then we have:
\begin{eqnarray*}
&&\|\Phi_t(z)-z\|_{\Ha_n,z} \leq \int_0^t \|H(\Phi_s(z))\|_{\Ha_n,z} \, ds \\
&\leq& \int_0^t \|\binom{H_1(\Phi_s(z))-2i \overline{\tilde{\Phi}_t}^T \tilde{H}(\Phi_s(z))}{0}\|_{\Ha_n,z} \, ds + \int_0^t \|\binom{2i\overline{\tilde{\Phi}_t}^T \tilde{H}(\Phi_t(z)),}{\tilde{H}(\Phi_t(z)),}\|_{\Ha_n,z} \, ds  \\
&\underset{\eqref{eq:14}, \eqref{normorth0}}{\leq}& \int_0^t \frac{c}{ u_{\Ha_n}(z)^2} \, ds + \int_0^t 2\frac{\|\tilde{H}(\Phi_s(z))\|}{|u_{\Ha_n}(z)|}\sqrt{|u_{\Ha_n}(z)|+\|\tilde{\Phi}_t-\tilde{z}\|^2} \, ds \\
&\underset{\eqref{eq:11}, \eqref{ineq}}{\leq}&
 \int_0^t \frac{c}{ u_{\Ha_n}(z)^2} \, ds + \int_0^t \frac{c}{ |u_{\Ha_n}(z)|^{5/2}}\sqrt{|u_{\Ha_n}(z)|+ \frac{c^2s^2}{4|u_{\Ha_n}(z)|^{3}}} \, ds   \\
&=&   \frac{ct}{ u_{\Ha_n}(z)^2} + \int_0^t \frac{c}{ |u_{\Ha_n}(z)|^{2}}\sqrt{1+ \frac{c^2s^2}{4|u_{\Ha_n}(z)|^{4}}} \, ds \\
&\leq&  \frac{ct}{ u_{\Ha_n}(z)^2} + \int_0^t \frac{c}{|u_{\Ha_n}(z)|^{2}}\sqrt{1+ c^2s^2} \, ds  \\
 &=&  c \cdot \frac{t +\int_0^t \sqrt{1+c^2s^2}\, ds}{ u_{\Ha_n}(z)^2}.  
\end{eqnarray*}

The case $|u_{\Ha_n}(z)|\leq 1$ is treated similarly, compare with the proof of Theorem \ref{semigroup}, and we conclude that for every $t\geq 0,$ there exists $C>0$ such that $\|\Phi_t(z)-z\|_{\Ha_n}\leq \frac{C}{u_{\Ha_n}(z)^2}$ for all $z\in\Ha_n.$
Together with Remark \ref{autoinf}, this implies that $\Phi_t\in \P_n.$
\end{proof}

\begin{remark} Let $H:[0,\infty)\times \Ha_n\to \C^n$ be a $\K(\Ha_n, c)-$Herglotz vector field, i.e. $H(t, \cdot)\in\K(\Ha_n, c)$ for almost every $t\geq 0$ and $H$ satisfies certain regularity conditions, see Definition 1.2 in \cite{MR2887104}. In this case, one can solve the non-autonomous version of equation \eqref{cauchy}, namely the Loewner equation
\begin{equation}\label{Loewner}\frac{\partial\Phi_t(z)}{\partial t} = H(t, \Phi(t)), \quad \Phi_0(z)=z\in \Ha_n,\end{equation}
which gives a family $\{\Phi_t\}_{t\geq0}$ of univalent self-mappings of $\Ha_n$, see Theorem 1.4 in \cite{MR2887104}. A slight variation of the proof of Theorem \ref{thm:1} shows that $\Phi_t\in \P_n$ for all $t\geq 0$ also in this case.
\end{remark}

\begin{question} Let $f\in \P_1.$ In \cite{MR1201130}, Section 4, it is shown that there exists a $\K(\Ha, c)-$Herglotz vector field $H$ and a time $T\geq0$ such that $f=\Phi_T$, where $\{\Phi_t\}_{t\geq0}$ is the solution of equation \eqref{Loewner}.
 What can be said in the higher dimensional case? 
\end{question}

\subsection{On the behavior of iterates}

Let $F:\B_n\to \B_n$ be holomorphic. We say that $p\in\overline{\B_n}$ is the Denjoy-Wolff point of $F$ if $F^n\to p$ for $n\to\infty$ locally uniformly. The basic results about the behavior of the iterates $F^n$ for $n\to\infty$ can be found in \cite{MR1098711}, Chapter 2.2. In particular we have (Theorem 2.2.31):  \begin{equation}\label{denj}
\text{$F$ has a Denjoy-Wolff point on the boundary $\partial\B_n$ $\Longleftrightarrow$ $F$ has no fixed points.}
\end{equation}

Now let $f\in \P_n.$ For $n=1$, $f$ has the Denjoy-Wolff point $\infty$ if $f$ is not the identity: As $f$ is not an elliptic automorphism, the classical Denjoy-Wolff theorem implies that $f$ has a Denjoy-Wolff point. This point has to be $\infty$, e.g. because of Proposition \ref{basics} c).\\
Next we will show that this is also true in higher dimensions, provided that $f$ extends smoothly to the boundary point $\infty.$ There are different possible definitions of smoothness of $f$ near $\infty.$ We will use the following one: Let $H(z)=f(z)-z$, and denote by $G:\B_n\to\C^n$ the corresponding generator on $\B_n$, i.e. we have 
$H(z)=(C^{-1})'(C(z))\cdot G(C(z))$
and a small computation shows $$H_1(z)= -\frac{i}{2}(z_1+i)^2 \cdot G_1(C(z)).$$

Our smoothness condition will be that $G_1$ has a $C^3$-extension to $e_1,$ i.e. we can write
$$ G_1(z) = \sum_{\substack{k_1+...+k_n\leq 3 \\
k_1,...,k_n\geq0}}a_{k_1, ..., k_n}(z_1-1)^{k_1}\cdot z_2^{k_2}\cdot... \cdot z_n^{k_n} + \Landauo(\|z-e_1\|^3), $$
which translates to
$$H_1(z) = -\frac{i}{2}(z_1+i)^2 \cdot  \sum_{k_1+...+k_n\leq 3}a_{k_1, ..., k_n}\left(\frac{-2i}{z_1+i}\right)^{k_1}\cdot \left(\frac{2z_2}{z_1+i}\right)^{k_2}\cdot... \cdot \left(\frac{2z_n}{z_1+i}\right)^{k_n} + \Landauo(\|C(z)-e_1\|^3),$$
or
\begin{eqnarray}\begin{split}\label{smoothness}H_1(z) &= b_{0,...,0}\cdot (z_1+i)^2 + (z_1+i) \cdot  \sum_{k_1+...+k_n= 1}b_{k_1, ..., k_n} z_2^{k_2}\cdot... \cdot z_n^{k_n} \\
&+   \sum_{k_1+...+k_n= 2}b_{k_1, ..., k_n} z_2^{k_2}\cdot... \cdot z_n^{k_n} 
+ (z_1+i)^{-1} \cdot  \sum_{k_1+...+k_n= 3}b_{k_1, ..., k_n} z_2^{k_2}\cdot... \cdot z_n^{k_n}\\
 &+ \Landauo(|z_1+i|^{-1}\cdot \|(1, z_2, ..., z_n)\|^3)
\end{split}
\end{eqnarray}
for some coefficients $b_{k_1, ..., k_n}\in \C.$
\begin{theorem}
Let $f\in \P_n, f\not= \text{id},$ and assume that \eqref{smoothness} is satisfied. Then $\infty$ is the Denjoy-Wolff point of $f$.
\end{theorem}

\begin{proof}
Write $f(z)=z+H(z),$ where $H\in \K(\Ha_n, c)$ and $H=(H_1, \tilde{H}).$ Let $\gamma\in \C^{n-1}$. If we can show that the slice $h_\gamma(\zeta)=H_1(\varphi(\zeta))-2i\overline{\gamma}^T \tilde{H}(\varphi_\gamma(\zeta))$ has no zeros, then we are done:\\
 This implies that $H$ has no zeros because of \eqref{eq:7} and \eqref{orthogonal}. Hence, $f$ has no fixed points and \eqref{denj} implies that $f$ has a Denjoy-Wolff point. This point has to be $\infty$ because of Proposition \ref{basics} c).\\ 
Similarly to the proof of Theorem \ref{semigroup}, equation \eqref{eq:9}, we have
$$ \|\tilde{H}(z)\| \leq \frac{c}{2|u_{\Ha_n}(z)|^{3/2}},$$
and thus $$\|\tilde{H}(\varphi_\gamma(\zeta))\|\leq \frac{c}{2|u_{\Ha_n}(\varphi_{\gamma}(\zeta))|^{3/2}} =
 \frac{c}{2 \Im(\zeta)^{3/2}}.$$
Consequently, $\lim_{y\to \infty} y |\overline{\gamma}^T \tilde{H}(\varphi_\gamma(iy))|=0$. On the other hand, we know from Proposition \ref{subset} that $h_\gamma\in \K(\Ha, c)$ which implies (see Remark \ref{rmk:1})
$$\limsup_{y\to \infty} y |h_\gamma(iy)|=\limsup_{y\to \infty} y |H_1(\varphi(iy))-2i\overline{\gamma}^T \tilde{H}(\varphi_\gamma(iy))| \leq c,$$
which gives us
\begin{equation}\label{boundedness}\limsup_{y\to \infty}  |iy \cdot H_1(\varphi_\gamma(iy))| \leq c.\end{equation}

Now we use the assumption of the smoothness of $H_1:$\\
Because of \eqref{boundedness}, all coefficients $b_{k_1,...,k_n}$ from \eqref{smoothness} with $k_1+...+k_n\leq 2$ have to be 0. Thus, $\lim_{y\to \infty} iy \cdot H_1(\varphi_\gamma(iy))=:K(\gamma)$ exists and is a polynomial in $\gamma=(\gamma_2, ..., \gamma_n):$
$$K(\gamma) =  \sum_{k_1+...+k_n= 3}b_{k_1, ..., k_n} \gamma_2^{k_2}\cdot... \cdot \gamma_n^{k_n}. $$ As $K(\gamma)$ is bounded, it has to be constant. \\
If $K(\gamma)\equiv 0,$ then all slices of $H$ are zero, hence $H=0$ by Remark \ref{rem2} and $f$ is the identity, a contradiction.\\
Hence $K(\gamma)$ is a non-zero constant and $h_\gamma(\zeta)$ is not identically zero, which implies (e.g. by using the representation \eqref{measure}) that $h_\gamma(\zeta)$ has no zeros.

\end{proof}

\begin{question}
Is $\infty$ the Denjoy-Wolff point for every $f\in \P_n$?
\end{question}

\section{Appendix}

Here we prove Lemma \eqref{formulas0}:\\

\begin{textit}
Let $a\in \C, p,v\in \C^{n-1}$ and $z=(z_1, \tilde{z})\in \Ha_n.$ Then the following formulas hold:\end{textit}

\begin{equation}\label{normproj0}\left\|\binom{a}{0}\right\|_{\Ha_n, z} = \frac{|a|}{|u_{\Ha_n}(z)|},
\end{equation}
\begin{equation}\label{normorth0}
\left\|\binom{2i \overline{p}^T v}{v}\right\|_{\Ha_n, z}= 2\frac{\sqrt{ \|v\|^2|u_{\Ha_n}(z)|+|\overline{(p-\tilde{z})}^T v|^2}}{|u_{\Ha_n}(z)|},
\end{equation}
\begin{equation}\label{orth0} \left\|\binom{a-2i \overline{\tilde{z}}^T v}{0} + \binom{2i \overline{\tilde{z}}^T v}{v}\right\|_{\Ha_n, z}^2=
 \left\|\binom{a-2i \overline{\tilde{z}}^T v}{0}\right\|_{\Ha_n, z}^2+
 \left\|\binom{2i \overline{\tilde{z}}^T v}{v}\right\|_{\Ha_n, z}^2 .\end{equation}

\begin{proof} We write $\tilde{z}=(z_2, ..., z_n), v=(v_2, ..., v_n), p = (p_2, ..., p_n).$\\

An explicit formula of the Kobayashi metric for the unit ball is given in \cite{MR2087579}, Theorem 3.4.\footnote{Note, however, that the Kobayashi metric in \cite{MR2087579} differs by a factor of 2 from the one we are using here.} It coincides with the Bergman metric and by using the Cayley map we get the following formula for the upper half-space: 
$$ \|w \|_{\Ha_n, z}^2 = w^T \cdot (g_{j,k})_{j,k} \cdot \overline{w},$$
where $w\in\C^n$ and $(g_{j,k})_{j,k} $ is an $n\times n$-matrix with
$$g_{j,k} = -4 \frac{\partial^2}{\partial z_j\, \partial \bar{z}_k} \log(\Im(z_1) - \sum_{l=2}^n |z_l|^2) ,$$ 
and we get for $j,k\geq 2$:
\begin{eqnarray*} g_{1,1} &=& \frac{1}{ u_{\Ha_n}(z)^2},\quad \quad g_{1,k}= \frac{2i z_k}{u_{\Ha_n}(z)^2}, \quad \quad g_{j,1}= \frac{-2i \overline{z_j}}{u_{\Ha_n}(z)^2},\\
g_{j,j} &=& 4\frac{\Im(z_1) - \sum_{l=2, l\not= j}^n |z_l|^2}{u_{\Ha_n}(z)^2},\quad \quad
g_{j,k} = \frac{4z_k \overline{z_j}}{u_{\Ha_n}(z)^2}, \; k\not=j.
\end{eqnarray*}

The formulas \eqref{normproj0} and \eqref{normorth0} are now straightforward calculations. We obtain

$$\|(a,0)\|_{\Ha_n, z} = \sqrt{ (a, 0) \cdot (g_{j,k})_{j,k} \cdot 
\overline{(a, 0)^T}}= \sqrt{ a \cdot g_{1,1} \cdot 
\overline{a}}=\frac{|a|}{ |u_{\Ha_n}(z)|},$$
and

\begin{eqnarray*}
&& u_{\Ha_n}(z)^2\cdot \|(2i \overline{p}^T v, v)\|_{\Ha_n, z}^2 =u_{\Ha_n}(z)^2\cdot (2i \overline{p}^T v, v^T) \cdot (g_{j,k})_{j,k} \cdot 
\overline{(2i \overline{p}^T v, v^T)^T} \\
&=& u_{\Ha_n}(z)^2\cdot (\sum_{j=2}^{n} g_{j,j} |v_j|^2 + g_{1,1} |2i \overline{p}^T v|^2
+  \sum_{j=2}^{n} g_{j,1} v_j \overline{2i\overline{p}^T v} +
\sum_{k=2}^{n} g_{1,k} \overline{v_j} 2i\overline{p}^T v + \sum_{j,k\geq 2, j\not=k}^{n} g_{j,k} v_j \overline{v_k} )\\
&=&
4\sum_{j=2}^{n}(\Im(z_1) - \|\tilde{z}\|^2)\cdot  |v_j|^2 + 4\sum_{j=2}^{n}|z_j|^2 \cdot |v_j|^2 \\
&+&  4\sum_{j,k\geq 2}^n p_j \overline{p_k}v_j \overline{v_k} - 4\sum_{j,k\geq 2}^{n} \overline{z_j}p_k v_j \overline{v_k} - 4\sum_{j,k\geq 2}^{n} z_j \overline{p_k} \overline{v_j} v_k + 4\sum_{j,k\geq 2, j\not=k}^{n} \overline{z_j}z_k v_j \overline{v_k}\\
&=& 
 4\|v\|^2 \cdot |u_{\Ha_n}(z)| + 4\sum_{j=2}^{n}z_j \overline{z_j} v_j \overline{z_j}\\
&+&   4\sum_{j,k\geq 2}^{n}
\left(p_j \overline{p_k}v_j \overline{v_k} - \overline{z_j}p_k v_j \overline{v_k} - z_j \overline{p_k} \overline{v_j} v_k \right) +  4\sum_{j,k\geq 2, j\not=k}^{n} \overline{z_j}z_k v_j \overline{v_k}
 \\
&=&  4\|v\|^2 \cdot |u_{\Ha_n}(z)| +  4\sum_{j,k\geq 2}^{n}
\left(p_j \overline{p_k}v_j \overline{v_k} - \overline{z_j}p_k v_j \overline{v_k} - z_j \overline{p_k} \overline{v_j} v_k +  \overline{z_j}z_k v_j \overline{v_k} \right) \\
&=&  4\|v\|^2 \cdot |u_{\Ha_n}(z)| +  4|\overline{(p-\tilde{z})}^T v|^2.
\end{eqnarray*}

For formula \eqref{orth0} we just need to show that $ (2i \overline{\tilde{z}}^T v, v^T) \cdot (g_{j,k})_{j,k} \cdot \overline{(a-2i \overline{\tilde{z}}^T v,0)}^T = 0.$ Indeed, we have 
$$  u_{\Ha_n}(z)^2\cdot (g_{j,k})_{j,k} \cdot \overline{(a-2i \overline{\tilde{z}}^T v,0)}^T = (\overline{a}+2i \tilde{z}^T \overline{v},
 -2i\overline{z_2} \overline{a}+4 \overline{z_2}\tilde{z}^T \overline{v}, ..., 
 -2i\overline{z_n} \overline{a}+4 \overline{z_n}\tilde{z}^T \overline{v})^T$$
and 
\begin{eqnarray*}&&(2i \overline{\tilde{z}}^T v, v^T)(\overline{a}+2i \tilde{z}^T \overline{v},
 -2i\overline{z_2} \overline{a}+4 \overline{z_2}\tilde{z}^T \overline{v}, ..., 
 -2i\overline{z_n} \overline{a}+4 \overline{z_n}\tilde{z}^T \overline{v})^T\\
&=& 
 2i \overline{a} \overline{\tilde{z}}^T v -4|\tilde{z}^T \overline{v} |^2 - 2i \overline{a} \overline{\tilde{z}}^T v 
+4|\tilde{z}^T \overline{v} |^2 = 0.
\end{eqnarray*}

\end{proof}

\newpage
\bibliographystyle{amsalpha}
\newcommand{\etalchar}[1]{$^{#1}$}
\providecommand{\bysame}{\leavevmode\hbox to3em{\hrulefill}\thinspace}
\providecommand{\MR}{\relax\ifhmode\unskip\space\fi MR }
\providecommand{\MRhref}[2]{%
  \href{http://www.ams.org/mathscinet-getitem?mr=#1}{#2}
}
\providecommand{\href}[2]{#2}

\end{document}